\documentclass[11pt]{article}
\usepackage{latexsym, amsmath, amssymb}

\usepackage{amsmath}
\usepackage{amsfonts}
\usepackage{amssymb}
\usepackage{amscd}
\usepackage{amsthm}
\usepackage{color}
\usepackage{verbatim}
\theoremstyle{plain}
\newtheorem{theorem}{Theorem}[section]
\newtheorem{proposition}[theorem]{Proposition}
\newtheorem{remark}[theorem]{Remark}
\newtheorem{lemma}[theorem]{Lemma}
\newtheorem{definition}[theorem]{Definition}
\newtheorem{corollary}[theorem]{Corollary}
\newtheorem{example}[theorem]{Example}

\DeclareMathOperator{\Ext}{Ext} \DeclareMathOperator{\End}{End}\DeclareMathOperator{\Hom}{Hom}

\def\C{\mathbb C}

\def\Q{\mathbb Q}

\newcommand\sA{{\mathcal A}}
\newcommand\sB{{\mathcal B}}
\newcommand\sC{{\mathcal C}}

\newcommand\sE{{\mathcal E}}
\newcommand\sF{{\mathcal F}}
\newcommand\sG{{\mathcal G}}

\newcommand\sK{{\mathcal K}}
\newcommand\sL{{\mathcal L}}

\newcommand\sO{{\mathcal O}}
\newcommand\sP{{\mathcal P}}
\newcommand\sQ{{\mathcal Q}}

\newcommand\sU{{\mathcal U}}
\newcommand\sV{{\mathcal V}}

\newcommand\sZ{{\mathcal Z}}

\def\pee#1{\hbox{$ {\mathbf P}^{#1}$}}
  \def\peen{\hbox{$ {\mathbf  P}^n$}}

  \def \tab#1{\kern #1 truein}

  \def\OP#1{\hbox{${\cal O}_{{\mathbf P}^{#1}}$}}

  \def\E{\hbox{${\cal E}$}}
  \def\F{\hbox{${\cal F}$}}
  
  \def\A{\hbox{${{\cal A}}$}}
  
  \def\C{\hbox{${\cal C}$}}

  \def\Q{\hbox{${\cal Q}$}}

  \def\O#1{\hbox{${\cal O}_{#1}$}}  

\begin{document}
\title{Horrocks Correspondence on ACM Varieties}
\author{F. Malaspina\thanks{Supported by the framework of PRIN 2010/11 \lq Geometria delle varieta' algebriche\rq, cofinanced by MIUR and partially supported by GNSAGA of Indam (Italy)} \ and A.P. Rao
\vspace{6pt}\\
 {\small  Politecnico di Torino}\\
{\small\it  Corso Duca degli Abruzzi 24, 10129 Torino, Italy}\\
{\small\it e-mail: francesco.malaspina@polito.it}\\
\vspace{6pt}\\
{\small   University of Missouri-St. Louis}\\
 {\small\it  St Louis, MO 63121, United States}\\
{\small\it e-mail: raoa@umsl.edu}} \maketitle \def\thefootnote{} 

\begin{abstract}We describe a vector bundle $\sE$  on a smooth $n$-dimensional ACM variety in terms of its
cohomological invariants $H^i_*(\sE)$, $1\leq i \leq n-1$, and certain graded modules of ''socle elements'' built from $\sE$.
In this way we give a generalization of the Horrocks correspondence. We prove existence theorems where we construct vector bundles
from these invariants and uniqueness theorems where we show that these data determine a bundle up to isomorphisms.  The cases of
the quadric hypersurface in $\mathbb P^{n+1}$ and the Veronese surface in $\mathbb P^5$ are considered in more detail.
\end{abstract}

\begin{section}*{Introduction}

In his fundamental paper \cite{Ho}, Horrocks described all vector bundles on projective space $\peen$ in terms of their
intermediate cohomology modules. He described these cohomology modules using what he called a $\frak Z$-complex and showed that
the category of vector bundles modulo stable equivalence was equivalent to the category of all $\frak Z$-complexes modulo exact
free complexes. In particular, this gives the well-known Horrocks criterion for a vector bundle to be a sum of line bundles in
terms of the vanishing of its intermediate cohomology. His results were reformulated by Walters (\cite{Wa}) into the language
of derived categories and extended to sheaves by Coanda (\cite{Co}). Beilinson  (\cite{Beil}) described the derived category of
sheaves on a projective space using complexes built from an ``exceptional sequence'' $\{ \OP n(1-n),\dots, \OP n(-1),\OP n \}$
of line bundles on $\peen$, and Kapranov (\cite{Kap}) gave a similar description for smooth quadric hypersurfaces by enlarging
the sequence to include the spinor bundles $\Sigma$ of the quadric. Ancona and Ottaviani (\cite{A-O}) used these methods to
extend the Horrocks splitting criterion to quadrics, with a theorem that a vector bundle $\sE$ on a quadric $\Q_n$ (of
dimension $n$) is a sum of line bundles if and only if $\sE$ has its intermediate cohomology modules $H^i_*(\sE)$ all zero for
$1\leq i \leq n-1$ and also $H^{n-1}_*(\sE\otimes \Sigma)=0$ for the spinor bundles $\Sigma$.

In this paper, we copy Horrocks' method on a smooth ACM subvariety $X$ of projective space. Given a vector bundle $\sE$ on $X$,
we construct a
$\frak Z$-complex of free $A$-modules (where $A$ is the coordinate ring of $X$). The zeroth syzygy of this complex, when
sheafified, gives a vector bundle $\sF$ on $X$ which we call an Horrocks data bundle for $\sE$, since it comes with a map
$\sF \to \sE$ which is an isomorphism on intermediate cohomology modules. When the map is injective, the quotient is
some ACM bundle on $X$.

These methods of Horrocks provide for ACM varieties rings a vector bundle version of results of Auslander and Bridger
(\cite{A-B}), who studied structure theorems for
modules of finite Gorenstein dimension over a commutative ring and showed that they are projectively equivalent to an
extension of a module of zero Gorenstein dimension by a module of finite projective dimension (see also \cite{M-P-S}, Ch3,
Proposition 8). In his unpublished 1986 preprint
\cite{Buch}, Buchweitz proves a similar result for strongly Gorenstein (non-commutative) rings, where to any
finitely generated module, he shows that it fits into a short exact sequence between two modules which he (and Auslander) calls
a maximal Cohen-Macaulay approximation to the module and  a hull of finite projective dimension. We will see that
the graded $A$-module $F$ of global sections of the Horrocks data bundle $\sF$ will have $F^\vee$ of finite projective dimension.

With this  natural extension of Horrocks' arguments to an ACM variety, we give a generalization of the Horrocks correspondence
in Section 1.
Our goal in looking at a Horrocks correspondence on $X$ is to look for cohomological invariants that determine $\sE$. We will
take the Horrocks data bundle as encoding all the intermediate cohomology for $\sE$ and view it as one of the invariants.
So we will study the bundles $\sE$ with a fixed (minimal)
Horrocks data bundle $\sF$. While for the map $\sF \to \sE$, the induced map of first cohomology modules $H^1_*(\sF) \to
H^1_*(\sE)$ is an isomorphism, for various irreducible ACM bundles $\sB$ on $X$, the map $H^1_*(\sF\otimes \sB^\vee) \to
H^1_*(\sE\otimes \sB^\vee)$ may have a kernel. These kernels will give more cohomological invariants and we
will call them modules of $\sB$-socle elements. In Theorems \ref{iso-thm} and \ref{iso-thm2}, we see how these invariants
determine
$\sE$ up to direct sums of ACM bundles. We also give a splitting criterion for the bundle $\sE$ to be a sum of
line bundles restricted from projective space. What is lacking in Section 1 is an  understanding of which modules of
socle elements are obtained from a vector bundle for a general ACM variety.

In Section 2 we describe  the case of quadrics, on which ACM bundles are well understood due to Kn\"{o}rrer (\cite{Kn}).
In particular, for the spinor bundles $\Sigma_i$ on a quadric $\sQ_n$, modules of $\Sigma_i$-socle elements of an
Horrocks data bundle $\sF$ are just graded vector spaces.  We  show that a vector bundle $\sE$ exists
for each choice of Horrocks data bundle $\sF$ and vector spaces $V_i$ of $\Sigma_i$-socle elements of $\sF$, and that two
vector bundles with the same data of $\sF, V_i$ (up to obvious isomorphisms) are isomorphic up to direct sums of ACM bundles.
In this way we generalize the results obtained in \cite{M-R} on $\sQ_2$.

In the last section we deal with the Veronese surface ${\sV}\subset\mathbb P^5$. The study of vector bundles on $\sV$ is
trivial by Horrocks if we view $\sV$ as $\mathbb P^2$. But as another illustration of
the methods, it is an interesting example of an arithmetically Cohen-Macaulay
embedding which is not arithmetically Gorenstein and for which the ACM bundles are easy to handle.

\end{section}

\section{Horrocks data bundles on ACM Varieties}
 Let $X$ be a smooth ACM variety of dimension $n$ in $\mathbb P^{n+r}$ over a field $k$.  For any sheaf $\sB$ on $X$, $H^i_*(\sB)$ will denote
 $\oplus_{l\in \mathbb Z}H^i(X, \sB(l))$.  The coordinate ring of $X$, $A=H^0_*(\O{X})$, is a noetherian Cohen-Macaulay graded
 $k$-algebra. $H^i_*(\sB)$ is a graded module
 over $A$. Let \textbf{M} be the
 category of graded, finitely generated $A$-modules and graded homomorphisms. Any finitely generated projective graded
 $A$ module  has the form $\oplus_i A(a_i)$ for some shifts $a_i\in \mathbb Z$ in grading, and will be called a free
 $A$-module.
 Let $\mathfrak P \subset
 \textbf{M}$ be the full subcategory of finitely generated free $A$-modules.
 $\mathcal C^-(\textbf M)$ (respectively $\mathcal C^-(\mathfrak P)$) will denote the category of all complexes, bounded above,
 of objects in $\textbf M$ (resp. $\mathfrak P$), where morphisms are  maps between two complexes.
 Since \textbf{M} has enough projectives, given a complex $C^\centerdot$ of objects in \textbf{M}, bounded above,
 one can find a free resolution:
\textit{ie.} a complex $P^\centerdot$ in $\mathcal C^-(\mathfrak P)$ with a quasi-isomorphism $P^\centerdot \to C^\centerdot$.

Let $\sE \in \textbf{VB}$ be an object in the category of vector bundles of finite rank on $X$. $H^i_*(\sE)$ is an $A$-module
of finite length for $1\leq i\leq n-1$. A vector
bundle will be called free if it has the form $\oplus_i \sO_X(a_i)$.  A vector bundle $\sE$ will be
called ACM (arithmetically Cohen-Macaulay) if $H^i_*(\sE)=0$ for all $1\leq i\leq n-1$. Since $X$ is ACM, every free bundle is
ACM. By Serre duality, the line bundle $\omega_X$ is an ACM line bundle.

Given $\sE$, let $E$ denote the graded $A$-module $H^0_*(\sE) $. Denoting duals by $^\vee$ in the categories \textbf{VB} and
\textbf{M}, we have $H^0_*(\sE^\vee) \cong (H^0_*(\sE))^\vee $.
Following Horrocks, we choose a resolution of $H^0_*(\sE^\vee)$ by finitely generated free modules:
\begin{equation}\label{Ho1} \dots \to C^{3\vee} \to C^{2\vee} \to C^{1\vee} \to C^{0\vee} \to H^0_*(\sE^\vee)
\to 0.\end{equation} In \cite{Ho}, this could be chosen as a finite resolution, but in our case, it may be infinite. However,
if $K = \ker(C^{n-2\vee} \to C^{n-3\vee})$, then $\sK$ is an ACM vector bundle on $X$ where $\sK=\tilde K$ is the sheaf obtained
from $K$.
Replacing the terms up to and including $C^{n-1\vee}$ by $K$ and dualizing, we get the complex,

\begin{equation}\label{Ho2} C^{\centerdot}_{\{0,n\}}:  0 \to C^0 \xrightarrow {\delta^1_{C^\centerdot}} C^1 \xrightarrow
{\delta^2_{C^\centerdot}} C^2 \xrightarrow {\delta^3_{C^\centerdot}} \dots\xrightarrow {\delta^{n-2}_{C^\centerdot}} C^{n-2}\to
K^\vee \to 0.\end{equation}

The exact sequence (\ref{Ho1}), when sheafified, gives an exact sequence of vector bundles, and its  dual gives the exact
sequence of vector bundles
\begin{equation}\label{Ho3} 0 \to \sE \to \tilde C^0 \xrightarrow {\delta^1_{C^\centerdot}} \tilde C^1 \xrightarrow
{\delta^2_{C^\centerdot}} \tilde C^2 \xrightarrow {\delta^3_{C^\centerdot}} \dots\xrightarrow {\delta^{n-2}_{C^\centerdot}}
\tilde C^{n-2}\to \sK^{\vee} \to 0.\end{equation}

From this it becomes evident that $E = H^0_*(\sE)$ is given as $H^0(C^\centerdot_{\{0,n\}})$, and $ H^i_*(\sE)=
H^i(C^\centerdot_{\{0,n\}})$ for $i=1,\dots n-1$ (where $C_{\{0,n\}}^{n-1}$ is understood to refer to $K^\vee$).

$E$ itself has a free resolution (again possible infinite). Splice $ C^\centerdot_{\{0,n\}}$ with a  free resolution
$L^\centerdot$ of $E$ and call the resulting complex $C^\centerdot$. The complex $C^\centerdot$ is bounded above and has the
property that $H^i(C^\centerdot)= H^i_*(\sE)$ for $i=1,\dots n-1$ and equals $0$ for other values of $i$.

Choose a free resolution $ P^\centerdot$ in $\mathcal C^-(\mathfrak P)$ of $C^\centerdot$.
$$\begin{array}{cccccccccccccccccc}

P^{\centerdot}: & \dots\to &P^{-2}&\to &P^{-1}& \to &P^0& \xrightarrow {\delta^1_{P^\centerdot}}& P^1 &\xrightarrow
{\delta^2_{P^\centerdot}}&
\dots\xrightarrow {\delta^{n-2}_{P^\centerdot}}&  P^{n-2}&\to &P^{n-1}& \to 0\\

 & &\downarrow& &\downarrow& &\downarrow& & \downarrow & &
&  \downarrow& &\downarrow& \\

C^{\centerdot}: & \dots\to &L^{-2}&\to & L^{-1}& \to &C^0& \xrightarrow {\delta^1_{C^\centerdot}}& C^1 &\xrightarrow
{\delta^2_{C^\centerdot}}&
\dots\xrightarrow {\delta^{n-2}_{C^\centerdot}}&  C^{n-2}&\to &K^\vee& \to 0\\

\end{array}$$
 Then $P^\centerdot$ is an element
in $\mathcal C^-(\mathfrak P)$ with the property that $H^i(P^\centerdot)$ is an $A$-module of at most finite length for $1\leq
i \leq n-1$, and is zero for other $i$. In Horrocks \cite{Ho}, the bounded version of such a free complex was called a $\frak
Z$-complex, while Walters (\cite{Wa}) calls the category of such complexes $\textbf{FinL}(\mathfrak P)$. In our setting, we
will call it an Horrocks data complex and use the notation of Walters (\cite{Wa}). We also define an ``Horrocks data bundle''
for each such Horrocks data complex:

\begin{definition}
$\textbf{FinL}^{-}(\mathfrak P)$ is the full subcategory of all complexes $P^\centerdot$ in $\mathcal C^-(\mathfrak P)$ with
the property that  $H^i(P^\centerdot)$ is an $A$-module of at most finite length for $1\leq i \leq n-1$, and is zero for other
$i$.  A complex $P^\centerdot$ in $\textbf{FinL}^{-}(\mathfrak P)$ will be called an Horrocks data complex. For such a complex,
let $F= \ker [\delta^1_{P^\centerdot}:P^0 \to P^1] $. Then the sheaf $\sF= \tilde F$  will be called an Horrocks data bundle on
$X$.
\end{definition}

It should be clear that the above $\sF$ is a vector bundle on $X$ with the property that $H^i_*(\sF) = H^i(P^\centerdot)$ for
$1\leq i \leq n-1$. Horrocks (\cite{Ho} Theorem 7.2) shows that $F^\vee$ has finite free resolution.

\begin{lemma}{\bf{(Horrocks)}} $F^\vee$ has a finite free resolution.
\end{lemma}
\begin{proof} Horrocks' proof cited above is when $A$ is a regular ring, but remains valid when $A$ is Cohen-Macaulay.
\end{proof}

Since the module of global sections of a non-free ACM bundle and of its dual bundle on $X$ have infinite projective dimension
over $A$, it follows that an Horrocks data bundle $\sF$ can have no non-free ACM bundle or its dual as a summand.

Since any $P^\centerdot$ in $\mathcal C^-(\mathfrak P)$ decomposes as $M^\centerdot \oplus L^\centerdot$, where $M^\centerdot$
is a minimal free complex and  $L^\centerdot$ is an acyclic free complex, we get $\sF = \sF_{min}\oplus \sL$ where $\sF,
\sF_{min}$, and $\sL$ are the Horrocks data bundles corresponding to $P^\centerdot,M^\centerdot$, and $L^\centerdot$
respectively. $\sL$ is a free bundle and $\sF_{min}$ will be called a ``minimal'' Horrocks data bundle. The following
isomorphism theorem on projective space can be found in \cite{Ho} Theorem 7.5, Proposition 9.5 or \cite{Wa} Lemma 2.11.

\begin{proposition}\label{min-F} Let $\sigma: \sF \to \sF'$ be a homomorphism between two minimal Horrocks data bundles on $X$
such that $\sigma$ induces isomorphism $H^i_*(\sF) \to H^i_*(\sF')$ for $1\leq i \leq n-1$. Then $\sigma$ is an isomorphism.
\end{proposition}
\begin{proof} The proofs cited above work in our ACM setting as well.
\end{proof}

Returning to the vector bundle $\sE$, let $ P^\centerdot$ be a free resolution of $C^\centerdot$ as described above. Let
$P^\centerdot_{\geq 0}$ denote the na\"{i}ve truncation of $P^\centerdot$ at the zeroth term. We get the induced homomorphism
of complexes
$$  P^\centerdot_{\geq 0}  \rightarrow  C^{\centerdot}_{\{0,n\}}. $$
For $F$  defined as $\ker \delta^1_{P^\centerdot}$, there is an induced homomorphism $F \to E$. For the Horrocks data
bundle $\sF= \tilde F$, we get a homomorphism $\beta: \sF \to \sE$  which induces  isomorphisms $H^i_*(\sF) \to H^i_*(\sE)$ for
$1\leq i \leq n-1$. Hence any vector bundle $\sE$ has an ``Horrocks datum'' as defined below:

\begin{definition} Let $\sE$ be a vector bundle on $X$. A pair $(\sF, \beta)$ will be called an Horrocks datum for $\sE$
if $\sF$ is an Horrocks data bundle and $\beta$ is a homomorphism $\beta:\sF \to \sE$ which induces isomorphisms $H^i_*(\sF)
\to H^i_*(\sE)$ for $1\leq i \leq n-1$.
\end{definition}

In dual form, the $A$-module $F^\vee$ in the definition above has been called by Auslander as a ``hull of finite projective
dimension'' for $E^\vee$, in his definition of maximal Cohen-Macaulay approximations for a module (\cite{Buch}). We use the
notation: Horrocks data bundle for $\sE$, since $\sF$ encodes all the intermediate cohomology data of $\sE$.

\begin{theorem}\label{Horrocks data}
\begin{enumerate}
\item Let $\sE_1, \sE_2$ be vector bundles on $X$ with Horrocks data $(\sF_1,\beta_1), (\sF_2, \beta_2)$ respectively.
 Let  $\sigma: \sE_1 \to \sE_2$ be a homomorphism. Then there is a free bundle $\sZ$ and a commuting square
 \[\begin{CD}
&\sF_1& & \rightarrow &\sF_2 &\oplus \sZ& \\
      &\downarrow &{\beta_1} &        &&\downarrow&{(\beta_2,*)} \\
&\sE_1& & \xrightarrow{\sigma}  &&\sE_2 &
\end{CD}
\]
\item If $H^0_*(\beta_2): H^0_*(\sF_2) \to H^0_*(\sE_2)$ is surjective, the free bundle $\sZ$ can be chosen to be zero.
\end{enumerate}
\end{theorem}
\begin{proof}  It is straightforward to see that the construction of the complex $C^\centerdot$ out of the vector bundle $\sE$
is functorial in the sense that given $\sigma: \sE_1 \to \sE_2$, there is induced a morphism from $C^\centerdot_1 \to
C^\centerdot_2$ with the property that the homomorphisms $H^i(C^\centerdot_1) \to H^i(C^\centerdot_2)$ for $1\leq i \leq n-1$
coincide with $H^i(\sigma):H^i_*(\sE_1) \to H^i_*(\sE_2)$. In the special case of $\beta_k: \sF_k \to \sE_k$, an Horrocks
datum, we get a quasi-isomorphism $P^\centerdot_k \to \C^\centerdot_k$, where $P^\centerdot_k$ is the Horrocks data complex
associated to $\sF_k$, so that $P^\centerdot_k \to \C^\centerdot_k$ is a free resolution of $\C^\centerdot_k$. Now given a
morphism of complexes $C^\centerdot_1 \to C^\centerdot_2$, we can lift the morphism to their free resolutions, after adding a
free acyclic complex to $P^\centerdot_2$. This gives the commuting square of part 1.  The proof of part 2 is elementary.
\end{proof}

The following theorems \ref{gamma} and \ref{eta} are to be found in more general form in \cite{Buch} as the ``Syzygy Theorem
for Gorenstein Rings''. The diagram on Theorem \ref{diag} below is Buchweitz's octahedron (\textit{loc cit}, 5.3.1).

\begin{theorem}{\bf{($\gamma$ sequence for $\sE$)}}\label{gamma} Let $\sE$ be a vector bundle on $X$, $(\sF,\beta)$ an
Horrocks datum for $\sE$. From the Horrocks data complex $P^\centerdot$ for $\sF$, consider the exact sequence $\Psi:0 \to \sF
\to \sP^0 \to \sG \to 0$, where $\sP^0 = \tilde P^0$ and $\sG = \tilde G$ with $G = \ker \delta^2_{P^\centerdot}$. We define
$\gamma$ as the push-out of $\Psi$ by $\beta$,
$$\begin{CD}
\Psi:     &0 \to &\sF             &\rightarrow  &\sP^0   &\rightarrow &\sG & \to 0 \\
          &      &\downarrow\beta &             & \downarrow &            & \|     &       \\
\gamma: &0 \to &\sE               &\xrightarrow{f}  &\sA   &\xrightarrow{g}  &\sG &\to 0
\end{CD}.$$
The following hold:
\begin{enumerate}

\item Given two bundles $\sE_1,\sE_2$, a morphism $\sigma: \sE_1 \to \sE_2$, and Horrocks data $(\sF_1, \beta_1), (\sF_2,
\beta_2)$ for each bundle, we obtain a commuting box of short exact sequences (using obvious notation)
$$ \begin{CD}
&\Psi_1& & \rightarrow &\Psi_2 &\oplus \lambda & \\
      &\downarrow &{\beta_1} &        &&\downarrow&{(\beta_2,*)} \\
&\gamma_1& & \xrightarrow{\sigma}  &&\gamma_2 &
\end{CD}
$$
where $\lambda$ is a short exact sequence $0 \to \sZ \to \sZ \to 0 \to 0$ of free bundles. If $H^0_*(\beta_2)$ is surjective
onto  $H^0_*(\sE_2)$, $\lambda$ may be taken to be zero. \item $H^{n-1}_*(\sG)=0,$ and $\sA$ is an ACM bundle on $X$. \item Up
to a short exact sequence $0 \to 0 \to \sZ \to \sZ \to 0$ of free bundles, the sequence $\gamma$ depends only on $\sE$ and not
on the choice of Horrocks datum.
\end{enumerate}
\end{theorem}

\begin{proof}
\begin{enumerate}
\item $\sigma$ lifts to a map $\sF_1 \to \sF_2 \oplus \sZ$ to give a commuting square, by the Theorem \ref{Horrocks data}.
$\sF_2 \oplus \sZ$ is an Horrocks data bundle for the Horrocks data complex where $P^0$ is replaced by $P^0 \oplus \sZ$ but
with the same bundle $\sG_2$. It is easy to see that the map $\sF_1 \to \sF_2 \oplus \sZ$ extends to a map of sequences $\Psi_1
\rightarrow \Psi_2 \oplus \lambda$. The push-outs of $\Psi_2$ and $\Psi_2 \oplus \lambda$ give the same sequence $\gamma_2$.
Lastly, since we have a commuting square from the first line of the proof, the pushouts of $\Psi_1$ and $\Psi_2 \oplus \lambda$
give a commuting box of exact sequences.
 \item By construction, $H^{n-1}_*(\sG) = H^n(P^\centerdot) =0$. Since we have
isomorphisms $H^i_*(\sG)\cong H^{i+1}_*(\sF) \cong H^{i+1}_*(\sE)$ for $1\leq i \leq n-2$ and $H^0_*(\sG)\twoheadrightarrow
H^{1}_*(\sF) \cong H^{1}_*(\sE)$, we conclude that $\sA$ is ACM. \item The last item follows from the first part when we apply
the previous theorem to the identity morphism from $\sE$ to $\sE$. Indeed, the theorem, together with Theorem \ref{min-F},
shows that any two Horrocks data bundles for $\sE$ are stably equivalent.
\end{enumerate}
\end{proof}

\begin{theorem}{\bf{($\eta$ sequence for $\sE$)}}\label{eta} Let $(\sF, \beta)$ be a  Horrocks datum for the bundle $\sE$
such that $H^0_*(\beta)$ is surjective. We define the $\eta$ sequence for $\sE$ to be (where $\sK$ is the kernel bundle)
$$ 0 \to \sK \to \sF \xrightarrow{\beta} \sE \to 0.$$
The following hold:
\begin{enumerate}
\item $\sK$ is an ACM bundle. \item $\eta$ is determined by $\sE$ up to a short exact sequence $0 \to \sZ \to \sZ \to 0\to 0$
of free bundles. \item Given a morphism $\sigma: \sE_1 \to \sE_2$, there is an induced morphism of short exact sequences
$\eta_1 \to \eta_2$.
\end{enumerate}
\end{theorem}

\begin{proof}The proof is easy. We just mention that the induced map $\eta_1 \to \eta_2$ depends on the choice of a map
from $\sF_1$ to $\sF_2$ that lifts $\sigma$ (as obtained from Theorem \ref{Horrocks data}).
\end{proof}

\begin{theorem}\label{diag}{{\bf(diagram of $\sE$)}} Let $(\sF, \beta)$ be a  Horrocks datum for the bundle $\sE$ such that
$H^0_*(\beta)$ is surjective. The $\gamma$ and $\eta$ sequences of $\sE$ fit into a \underbar{diagram} for $\sE$
$$\begin{array}{ccccccc}
&  & &0 & &0&   \\
&  &  &\downarrow & &\downarrow &    \\
&  & & \sK                &= &\sK &    \\
&  & & \downarrow\alpha & &\downarrow &     \\
& \Psi:\qquad &0 \to &\sF          &\xrightarrow{g} & \sP^0 & \xrightarrow {} \sG  \to 0\\
& & &\downarrow\beta & &\downarrow & \|    \\
& \gamma:\qquad & 0 \to &\sE& \xrightarrow{f} & \sA&\xrightarrow{} \sG \to 0\\
& & &\downarrow & &\downarrow &    \\
& & & 0 & &0&   \\
\\
& & & \eta & &\Delta&   \\
 \end{array}$$
Given a morphism $\sigma: \sE_1 \to \sE_2$, there is an induced map from the diagram of $\sE_1$ to the diagram of $\sE_2$.
\end{theorem}
\begin{proof} While the existence of the diagram is clear, the map from diagram of $\sE_1$ to the diagram of $\sE_2$ with
appropriate commuting boxes exists because the choice of a map from $\sF_1$ to $\sF_2$ that lifts $\sigma$ will determine
$\eta_1 \to \eta_2$ and then allows a choice of a map $\Psi_1 \to \Psi_2$. This now gives the commuting  box of short exact
sequences of Theorem \ref{gamma}.
\end{proof}

The following is a criterion for obtaining a map between two $\gamma$-sequences.

\begin{proposition}\label{gamma-map} Let $\sE, \sE'$ be two vector bundles with the same (minimal) Horrocks data bundle
$\sF_{min}$, and Horrocks data $(\sF_{min}, \beta), (\sF_{min},\beta')$. Let $\sB_1, \sB_2,\dots, \sB_k$ be the distinct
non-free irreducible ACM bundles (up to twists by $\sO_{X}(a)$) that appear as summands in the middle term $\sA_{\sE}$ of the
$\gamma$-sequence of $\sE$. For each $\sB_i$, let $V_i$ be the kernel of the map $H^1_*(\beta \otimes 1_{\sB_i^\vee})$ from
$H^1_*(\sF_{min}\otimes \sB_i^\vee) \to H^1_*(\sE\otimes \sB_i^\vee)$, and let $V_i'$ be the same with $\beta$ replaced by
$\beta'$. If $V_i \subseteq V_i'$ for all $i$, then there exists a map $\phi: \sE \to \sE'$ such that the $\gamma$-sequence of
$\sE'$ is the push out by $\phi$ of the $\gamma$-sequence for $\sE$.
\end{proposition}
\begin{proof}
Since the $\gamma$-sequences $\gamma, \gamma'$ are push-outs by $\beta, \beta'$ of the $\Psi$-sequence for $\sF_{min}$:
$$ \Psi:     0 \to \sF_{min} \rightarrow  \sP^0   \rightarrow \sG_{min}  \to 0, $$
in the commutative diagram
\[ \begin{CD}
\Hom (\sP^0, \E')\to  & \Hom(\sF_{min}, \sE') &\xrightarrow{\delta(\Psi)} &\Ext^1(\sG_{min}, \sE') &\xrightarrow{} \Ext^1(\sP^0, \sE')\\
& \uparrow \beta & & || &  \uparrow \\
& \Hom(\sE, \sE')  &\xrightarrow{ \delta(\gamma)} &\Ext^1(\sG_{min}, \sE') &\xrightarrow{} \Ext^1(\sA_{\sE}, \sE'), \\
\end{CD}
\]
it suffices to show that $\gamma' \in \Ext^1(\sG_{min}, \sE')$ maps to zero in $\Ext^1(\sA_{\sE}, \sE')$. For then there is an element
$\sigma \in \Hom(\sE, \sE')$ such that $\sigma \circ \beta$ differs from $\beta'$ by a map that factors through $\sP^0$.

Let $\rho: \sA_{\sE} \to \sG_{min}$ be the map occurring in the $\gamma$-sequence of $\sE$. Then under the connecting homomorphism
for $\gamma \otimes \sA_{\sE}^\vee$, $\rho$ maps to zero under $H^0_*(\sG_{min}\otimes \sA_{\sE}^\vee) \to
H^1_*(\E \otimes \sA_{\sE}^\vee)$. Hence under the connecting homomorphism of $\Psi \otimes \sA_{\sE}^\vee$, $\rho$ maps to the
the kernel of $H^1_*(\sF_{min}\otimes \sA_{\sE}^\vee) \to H^1_*(\sE \otimes \sA_{\sE}^\vee)$. By the assumption $V_i \subseteq V_i'$
for all $i$, $\rho$ also maps to the kernel of $H^1_*(\sF_{min}\otimes \sA_{\sE}^\vee) \to H^1_*(\sE' \otimes \sA_{\sE}^\vee)$. It
follows that the pullback of $\gamma'$ by $\rho$ splits, which was the desired result.
\end{proof}

This criterion leads to an isomorphism theorem on $X$:

\begin{theorem}\label{iso-thm}{\bf(Isomorphism theorem.)}Let $\sE, \sE'$ be two vector bundles on $X$, with the same minimal Horrocks
data bundle $\sF_{min}$ and Horrocks data $(\sF_{min}, \beta), (\sF_{min},\beta')$. Let $\sB_1, \sB_2,\dots, \sB_k$ be
the distinct non-free irreducible ACM bundles (up to twists by
$\sO_{X}(a)$) that appear as summands in either of the middle terms $\sA_{\sE}, \sA_{\sE'}$ of the $\gamma$-sequences of $\sE$, $\sE'$.
If for each $i$, the kernel of $H^1_*(\beta \otimes 1_{\sB_i^\vee})$ equals the kernel of $H^1_*(\beta' \otimes 1_{\sB_i^\vee})$
and if $\sE$ and $\sE'$ have no ACM summands, then $\sE \cong \sE'$.
\end{theorem}
\begin{proof}
If $\sF$ is free, $\sE, \sE'$ are ACM and the theorem does not apply. So we will assume that $\sF_{min}$ is a non-free minimal
Horrocks data bundle. By applying Proposition \ref{gamma-map}, there exists a homomorphism $\sigma: \sE \to \sE'$ and a
commutative diagram of $\gamma$-sequences
 \[ \begin{CD}
0 \to & \sE \to &\sA_{\sE} \to &\sG_{min} &\to 0\\
      & \downarrow \sigma & \downarrow \sigma_1 & ||&\\
0 \to & \sE' \to &\sA_{\sE'} \to &\sG_{min} &\to 0.
\end{CD}
\]

Tensor the diagram by $\sB^\vee$, where $\sB$ will stand for any of the distinct irreducible ACM bundles (up to twists by
$\sO_{X}(a)$) that appear as summands in $\sA_{\sE'}$, including the possible free line bundle $\sO_X$.
In the induced diagram of cohomology, we get

 \[ \begin{CD}
0 \to & H^0_*(\sE\otimes \sB^\vee) \to &H^0_*(\sA_{\sE}\otimes \sB^\vee) \to &H^0_*(\sG_{min}\otimes \sB^\vee) \to
&H^1_*(\sE\otimes \sB^\vee) &\to &H^1_*(\sA_{\sE}\otimes \sB^\vee)\\
      & \downarrow \sigma & \downarrow \sigma_1 & ||& \downarrow \sigma & & \downarrow \sigma_1\\
0 \to & H^0_*(\sE'\otimes \sB^\vee) \to &H^0_*(\sA_{\sE'}\otimes \sB^\vee) \to &H^0_*(\sG_{min}\otimes \sB^\vee) \to
&H^1_*(\sE'\otimes \sB^\vee) &\to &H^1_*(\sA_{\sE'}\otimes \sB^\vee).\\
\end{CD}
\]

The map $H^0_*(\sG_{min}\otimes \sB^\vee) \to H^1_*(\sE\otimes \sB^\vee)$ factors through $H^1_*(\sF \otimes
\sB^\vee)$, since the $\gamma$ is the pushout of $\Psi$ by $\beta$. The condition of equality of kernels for
$H^1_*(\beta \otimes 1_{\sB^\vee})$ and $H^1_*(\beta' \otimes 1_{\sB^\vee})$ implies that the kernel in
$H^0_*(\sG_{min}\otimes \sB^\vee)$ is the same for $\sE$ and $\sE'$. Therefore the mapping cone map $H^0_*(\sE'\otimes
\sB^\vee) \oplus H^0_*(\sA_{\sE}\otimes \sB^\vee) \to H^0_*(\sA_{\sE'}\otimes \sB^\vee)$ is surjective. Viewing
each summand $\sB$ of $\sA_{\sE'}$, the identity global section in $H^0(\sB \otimes \sB^\vee)$ is in the image
of this surjection. It cannot be in the image of $H^0_*(\sE'\otimes \sB^\vee)$ since $\sE'$ does not have $\sB$ as a
summand. Hence it is in the image of some $\sB'$ term in $\sA_{\sE}$. This forces $\sB'$ to equal $\sB$ and the map
$\sigma_1: \sA_{\sE} \to \sA_{\sE'}$ has to split over this $\sB$ term in $\sA_{\sE'}$.

It follows that $\sigma_1$ is a (split) surjection. Hence  $\sigma:\sE \to \sE'$ is onto. The roles of $\sE, \sE'$ can be
interchanged, showing that they are bundles of the same rank. Hence $\sigma: \sE \cong \sE'$.
\end{proof}

The following theorem is in the same vein, and extends Proposition \ref{min-F}:

\begin{theorem}\label{iso-thm2} Let $\sigma: \sE \to \sE'$ be a sheaf homomorphism between two vector bundles on $X$. Suppose that
$\sigma$  induces isomorphisms $H^i_*(\sE) \to H^i_*(\sE')$ for $1 \leq i \leq n-1$ and also for each non-free irreducible ACM
bundle $\sB$ appearing in $\sA_{\sE'}$, suppose that the induced map $H^1_*(\sE \otimes \sB^\vee) \to
H^1_*(\sE' \otimes \sB^\vee)$ is an isomorphism. Then $\sigma$ is a split surjection decomposing $\sE$ into $\sE' \oplus \sC$
where $\sC$ is an ACM bundle.
\end{theorem}
\begin{proof}
By Theorem \ref{Horrocks data}, $\sigma$ can be lifted to a map $\tilde\sigma: \sF_{min} \to \sF'_{min}$ of minimal
Horrocks data bundles. Since $H^i_*(\tilde\sigma)$ is an isomorphism for $1 \leq i \leq n-1$, $\tilde\sigma$ is an isomorphism.
So for convenience, we may assume that $\sF_{min} = \sF'_{min}$, and according to Theorem \ref{gamma}, $\sigma$ induces a map
of $\gamma$-sequences
 \[ \begin{CD}
0 \to & \sE \to &\sA_{\sE} \to &\sG_{min} &\to 0\\
      & \downarrow \sigma & \downarrow \sigma_1 & ||&\\
0 \to & \sE' \to &\sA_{\sE'} \to &\sG_{min} &\to 0.
\end{CD}
\]
For each $\sB$ appearing in $\sA_{\sE'}$, as in the proof of the previous theorem, after tensoring by $\sB^\vee$ we can look
at the diagram of cohomology. Since $H^1_*(\sE \otimes \sB^\vee) \to H^1_*(\sE' \otimes \sB^\vee)$ is an isomorphism,
the kernel in
$H^0_*(\sG_{min}\otimes \sB^\vee)$ is the same for $\sE$ and $\sE'$. The previous argument repeats to show that the
homomorphism $\sigma_1: \sA_{\sE} \to \sA_{\sE'}$ is a split surjection, with a kernel $\sC$ which is ACM. Hence
$\sigma: \sE \to \sE'$ is also a split surjection with kernel equal to $\sC$.
\end{proof}

Since the $A$-submodules $V_i$ play such an important role in the description of a bundle $\sE$, it is worthwhile to make
the following definition:
\begin{definition}\label{socl} Let $\sF$ be a sheaf on $X$ and $\sB$ an ACM bundle on $X$ with a minimal set of generators for $H^0_*(\sB)$
given by
$\oplus _j \sO_X(a_j) \to \sB \to 0$. The kernel of $H^1_*(\sF\otimes \sB^\vee) \to H^1_*(\sF\otimes \oplus_j \sO_X(-a_j))$ will
be called the $A$-module of $\sB$-socle elements for $\sF$ and denoted $H^1_*(\sF\otimes \sB^\vee)_{soc}$. A homogeneous element
in this kernel in degree $d$ will be a $\sB$-socle element in $H^1(\sF(d)\otimes \sB^\vee)$.
\end{definition}

\begin{remark}\label{socle-remark}
\end{remark}
\begin{enumerate}
\item For a vector bundle $\sF$, the module of $\sB$-socle elements for $\sF$ has finite length over the field $k$.
\item Suppose $\sB^\vee \to \sO_X(b)$ is any map. Then, for any sheaf $\sF$, a $\sB$-socle element in $H^1_*(\sF \otimes
\sB^\vee)$ maps to zero in $H^1_*(\sF(b))$, since $\sB^\vee \to \sO_X(b)$ factors through $\oplus_j \sO_X(-a_j)$.
\item Suppose $\sE$ is a bundle on $X$ with Horrocks datum $(\sF_{min}, \beta)$. Then for any ACM bundle $\sB$, the module
$V = \ker(H^1_*(\sF_{min} \otimes \sB^\vee) \to H^1_*(\sE\otimes \sB^\vee)$ consists of $\sB$-socle elements for $\sF_{min}$.
Indeed, the map $H^1_*(\sF_{min} \otimes \oplus_j \sO_X(-a_j)) \to H^1_*(\sE\otimes \oplus_j \sO_X(-a_j))$ is an isomorphism.
\end{enumerate}


\begin{example} \end{example}
As an example, any ACM variety $X$ with a non-degenerate embedding into $\mathbb P^N$ has a Horrocks data bundle given by $\Omega^1_
{\mathbb P}|_X$ with $H^1_*(\Omega^1_{\mathbb P}|_X)=k$ and with an exact sequence
$$ 0 \to \Omega^1_{\mathbb P}|_X \to \sO_X(-1)^{\oplus {N+1}} \to \sO_X \to 0.$$
For any ACM bundle $\sB$ on $X$, without free summands and with $\sB^\vee \hookrightarrow \oplus_j \sO_X(-a_j)$, consider the
diagram
 \[ \begin{CD}
 H^0_*(\sO_X \otimes \sB^\vee) \qquad \to & H^1_*(\Omega^1_{\mathbb P}|_X \otimes \sB^\vee) \\
       \downarrow  & \downarrow  \\
 H^0_*(\sO_X \otimes \oplus_j \sO_X(-a_j))  \to &H^1_*(\Omega^1_{\mathbb P}|_X \otimes \oplus_j \sO_X(-a_j)).
\end{CD}
\]
Then any minimal generator of the module $H^0_*(\sO_X \otimes \sB^\vee)$ maps to a non-generator in $H^0_*(\sO_X \otimes \oplus_j
\sO_X(-a_j))$, hence maps to zero in $H^1_*(\Omega^1_{\mathbb P}|_X \otimes \oplus_j \sO_X(-a_j))=\oplus_j k(-a_j)$. Thus the image
of $H^0_*(\sO_X \otimes \sB^\vee)$ in $H^1_*(\Omega^1_{\mathbb P}|_X \otimes \sB^\vee)$ is non-zero and consists of $\sB$-socle
elements for $\Omega^1_{\mathbb P}|_X$. So for any ACM bundle $\sB$ on $X$, without free summands, the Horrocks data bundle
$\Omega^1_{\mathbb P}|_X$ will have $\sB$-socle elements.

For a general ACM variety $X$, one would expect infinitely many families of non-isomorphic and irreducible ACM bundles; hence
this shows that even for a fixed Horrocks data bundle $\sF_{min}$, the number of bundles $\sE$ with Horrocks datum $(\sF_{min},
\beta_{\sE})$ would get out of control, especially with the construction given below. In later sections, we will limit our
attention to the quadric hypersurface and the Veronese surface, where there are only finitely many ACM bundles. In these
sections, we will be able to deal with arbitrary submodules of $\sB$-socle elements, instead of the entire $\sB$-socle module
of the rather crude theorem below.

\begin{theorem}{{\bf (Existence)}}{\label {existence}} Let $\sF_{min}$ be a minimal Horrocks data bundle on $X$, and let $\sB_1,\sB_2,\dots \sB_k$
a finite collection of irreducible, non-free ACM bundles on $X$. Suppose for each $i$, $V_i^{max}$ is a non-zero graded
vector sub-space of the $A$-module $H^1_*(\sF\otimes \sB_i^\vee)_{soc}$ that is generated by a collection of minimal generators
of the module. Then there is a vector bundle $\sE$ on $X$ with Horrocks datum $(\sF_{min}, \beta)$,
with $H^1_*(\sF_{min}\otimes \sB_i^\vee)_{soc} = \ker H^1_*(\beta \otimes 1_{\sB_i^\vee})$ for $1\leq i \leq k$.
\end{theorem}
\begin{proof}. Let $\sB = \oplus (V_i^{max}\otimes_k \sB_i)$. The data $V_i^{max}, 1\leq i \leq k$ can be viewed as a socle
element in $H^1_*(\sF_{min}\otimes \sB^\vee)$, hence gives an extension (that defines a bundle $\sE$)
$$ 0 \to \sF_{min} \xrightarrow{\beta} \sE \xrightarrow{\rho} \sB \to 0.$$
Since the element is a socle element, the pullback of the sequence under any map $\sO_X (b) \to \sB $ will split. Hence
$H^0_*(\rho)$ is surjective, giving $(\sF_{min}, \beta)$ the Horrocks datum for $\sE$.

By construction, the subspace $V_i^{max}\cdot I_{\sB_i}$ in $H^0_*(\sB \otimes \sB_i^\vee)$ maps isomorphically to $V_i^{max}
\subseteq H^1_*(\sF_{min}\otimes \sB_i^\vee)_{soc}$. Hence the image of the map of $A$-modules $H^0_*(\sB \otimes \sB_i^\vee) \to
H^1_*(\sF_{min}\otimes \sB_i^\vee)_{soc}$ is onto.
\end{proof}

\begin{remark}\end{remark}
\begin{enumerate}
\item The same construction can be done for arbitrary subspaces $V_i$ of $H^1_*(\sF\otimes \sB_i^\vee)_{soc}$. But then,
the bundle $\sE$ so constructed will have $\ker H^1_*(\beta \otimes 1_{\sB_i^\vee})$ containing the submodule generated by
$V_i$ without a precise knowledge of how much larger it is. Hence the Horrocks invariants of $\sE$ are not so recognizable.
\item In the above theorem, for the $\sE$ so constructed, it is possible to identify $A_{\sE}$ in the case when $X$ is arithmetically
Gorenstein, or when the dual
of each of the ACM bundles $\sB_i, 1\leq i \leq k$ is also ACM: since the $\gamma$-sequence of $\sE$ is the push-forward of the
$\Psi$-sequence for $\sF_{min}$, we get the exact sequence $0 \to \sP^0 \to \sA_{\sE} \to  \sB \to 0$ which is forced to split
with the extra hypotheses. Once the ACM bundles in $\sA_{\sE}$ are identified, it is possible to compare $\sE$ with other bundles
via the uniqueness theorems \ref{iso-thm}, \ref{iso-thm2}.
\item However, the theorem (and its proof) in the non-arithmetically
Gorenstein case, in addition to the shortcoming that it produces only maximal socle sub-modules,  is also too crude even to
allow a clear description of $\sA_{\sE}$. We will give an example later of a non-Gorenstein case
where such an identification of $\sA_{\sE}$ fails.\end{enumerate}

It is easy to obtain a splitting criterion for a vector bundle $\sE$ on $X$ to be free, which gives for example the criterion
for quadrics in \cite{A-O} that was cited in the introduction.. Once again, in the theorem below, note that the condition
invoking any ACM bundle is not very useful when there are too many ACM bundles on $X$. It is more interesting (see the proof
below) in the case
where the choices for $\sB$  are limited; for example, if one could limit the possible ACM bundles that
might appear as a summand in the diagram of $\sE$ .

\begin{theorem}{{\bf(a splitting criterion)}} Let $\sE$ be a vector bundle of rank $\leq r$  on $X$, a smooth ACM variety
of dimension $n$, such that $H^i_*(\sE^\vee)=0$ for $1 \leq i \leq \min\{r-1, n-1\}$ and also $H^{1}_*(\sE^\vee \otimes \sB) =
0$ for any ACM bundle $\sB$ on $X$. Then $\sE$ is free.
\end{theorem}
\begin{proof}
Now the $\eta$-sequence (Theorem \ref{eta}) of $\sE$,  $0 \to \sK \to \sF \to \sE \to 0$, gives an element in $H^1 (\sE^\vee
\otimes \sK)$ which is zero by hypothesis. Hence $\sK$ and $\sE$ are summands of $\sF$. Since $\sF$ is an Horrocks data bundle,
it can have no non-free ACM summand, so $\sK$ must be free. Thus $\sE$ itself is an Horrocks data bundle.

If $r\geq n$, $\sE^\vee$ is ACM. But the dual of an Horrocks data bundle has finite resolution, so $\sE^\vee$ must be free.

If $r<n$,  consider the sequence (\ref{Ho3}) with $\sE$ replaced by $\sE^\vee$. From the vanishing of cohomologies of
$\sE^\vee$, when we look at the complex of global sections of the sequence, we conclude
that the module $E^\vee$ is an $(r+1)^{th}$-syzygy, and $E^\vee$ has finite projective dimension since $\sE$ is an
Horrocks data bundle. By the Evans-Griffith syzygy theorem (\cite{E-G}), $\sE^\vee$ is free.
\end{proof}

\begin{remark} \end{remark}

 If $X$ is a smooth quadric hypersurface the above splitting criterion is also equivalent to Corollary 4.3. of
\cite{B-M}. In other varieties the criterion may be improved with a case by case analysis. For instance in a Grassmanniann
of lines, it is possible to recover Theorem 2.6 of \cite{am} and in multiprojective spaces it is possible to recover
Theorem 3.9. of \cite{bm2}

\section {Quadric Hypersurfaces}
 Let $\Q_n \subset \mathbb P^{n+1}$ be a  smooth quadric hypersurface. We will work over a field of
 characteristic  not two. The quadratic form defining $\Q_n$ descends to a quadratic form on the tangent bundle
 of $\Q_n$. Hence one can define spinor bundles on $\Q_n$ (\cite{Kar}).  Set $l:= \lfloor (n+1)/2\rfloor$.
 If $n$ is even, then
$\Q_n$ has two distinct spinor  bundles $\Sigma_1$ and $\Sigma_2$ of rank $2^{l-1}$. If $n$ is odd, then $\Q_n$ has a unique
spinor bundle, which we denote $\Sigma_1$, of rank $2^{l-1}$. Algebraic properties of these bundles were studied by Ottaviani
(\cite{Ot}) who obtained them using the geometry of the variety of all maximal linear subspaces of $\Q_n$ to construct
morphisms from $\Q_n$ to $G(2^{l-1},2^l)$. He shows that these spinor bundles on $\Q_n$ are ACM bundles. Kapranov (\cite{Kap})
showed how these bundles were crucial in describing the derived category of sheaves on the quadric. Meanwhile, Kn\"{o}rrer
(\cite{Kn}), classifying maximal Cohen-Macaulay modules over isolated quadratic hypersurface singularities, described these
bundles as the fundamental ACM bundles on $\Q_n$ (see \cite{B-G-S} for the interpretation of Kn\"{o}rrer's results in terms of
bundles). Kn\"{o}rrer's classification of ACM bundles on $\Q_n$ was proved also in \cite{A-O}.

We use a unified notation $\Sigma_i$ for spinor bundles on $\sQ_n$, where  for even $n$, $i$ can take on the values $1,2$,
 while if $n$ is odd, $i$ can be only $1$. We follow the notation of \cite{Kap}, whose  spinor bundles differ from those in
\cite{Ot} by a twist of $1$. Hence $\Sigma_i$ is generated by its global sections and $\Sigma_i(-1)$ has no sections.

We will call a bundle of the form $\Sigma_i(a)$ a twisted spinor bundle on  $\Q_n$. The fundamental theorem of Kn\"{o}rrer
\cite{Kn} is

\begin{theorem}{\bf(Kn\"{o}rrer)} Any ACM bundle on $\Q_n$ is a direct sum of line bundles and twisted spinor bundles.
\end{theorem}

The spinor bundles on $\Q_n$ satisfy some dualities (\cite{Ot}): When $n$ is odd or $n \equiv 0(\mod 4)$, $\Sigma_i^\vee \cong
\Sigma_i(-1)$, while if $n \equiv 2(\mod 4)$, $\Sigma_i^\vee \cong \Sigma_j(-1)$ where $j\neq i$.

In addition, the spinor bundles on $\Q_n$ satisfy canonical sequences. To further unify the notation, when $n$ is odd or when
$n \equiv 2(\mod 4)$ , define $i \mapsto \bar i$ to be the identity on indices, and when $n \equiv 0(\mod 4)$, define $i
\mapsto \bar i$ to be the transposition of the indices $1$ and $2$. With this notation, we have the canonical sequences

\begin{equation} \label{sp8}
   0 \to \Sigma_{\bar i} ^\vee \xrightarrow{u_i} \sO^{\oplus 2^l}  \xrightarrow{v_i} \Sigma_i \to 0
\end{equation}(see \cite{Ot} Theorem 2.8).  \\

In \cite{Ot} Lemma 2.7., Ottaviani proves that for any spinor bundle $\Sigma_i$, $\End(\Sigma_i)=H^0(\Sigma_i\otimes
\Sigma_i^\vee) =k$ and $\Hom(\Sigma_i, \Sigma_j)=0$ for $i\neq j$. Using this, and tensoring the sequence above with $\Sigma_i^\vee$,
we get $H^1(\Sigma_{\bar i}^\vee\otimes
\Sigma_i^\vee)=k$, where $Id_{\Sigma_i}$ maps to a generator of $H^1(\Sigma_{\bar i} ^\vee\otimes \Sigma_i^\vee)$. For
completeness, the following lemma is also easy to prove:
\begin{lemma}\label{ot} \begin{align} &H^1_*(\Sigma_{\bar i} ^\vee\otimes \Sigma_i^\vee) = k \\
                        &H^1_*(\Sigma_{j}^\vee\otimes \Sigma_i^\vee) = 0, \text{if } j\neq \bar i.\end{align}

\end{lemma}

Recall the definition of socle elements.

\begin{definition} Let $\sF$ be a sheaf on $\Q_n$.
The sequence dual to (\ref{sp8}) tensored by $\sF$ gives
$$0 \to  \sF\otimes\Sigma_{i}^\vee \rightarrow \sF \otimes \sO^{\oplus 2^{l}} \rightarrow
\sF\otimes\Sigma_{\bar i} \to 0,$$
and a natural map
$ H^1_*(\sF\otimes\Sigma_{i}^\vee)\to H^1_*(\sF\otimes \sO^{\oplus 2^{l}})$.\\
     An element in $H^1(\sF(d)\otimes\Sigma_{i}^\vee)$ will be called a $\Sigma_i $-socle element for $\sF$ in degree $d$ if
     it is annihilated
    by the map $H^1(\sF(d)\otimes\Sigma_{i}^\vee) \rightarrow  H^1_*(\sF\otimes \sO^{\oplus 2^{l}})$.
\end{definition}

The terminology ``socle'' comes from the case of a quadric surface studied in \cite{M-R}, where socle elements were
annihilated by multiplication by the forms lifted from one of the $\mathbb P^{1}$ factors of $\sQ_2$. We have extended this
terminology to all ACM bundles in Section 1.

\begin{lemma} \label{socle} Let $\sF$ be a sheaf on $\Q_n$.
 Let  $V$ be a finite-dimensional graded subspace consisting of $\Sigma_i$-socle elements in $H^1_*(\sF\otimes
\Sigma_{i}^\vee)$. Then there is a homomorphism $\alpha: V\otimes \Sigma_{\bar i}^\vee\to \sF$ such that $H^1_*(\alpha \otimes
1_{\Sigma_{i}^\vee})$ has image $V$.
\end{lemma}
\begin {proof}
 Consider the dual canonical sequence (\ref{sp8}) tensored
by $\sF$
$$0 \to  \sF\otimes \Sigma_{i}^\vee \rightarrow \sF\otimes \sO^{\oplus 2^{l}}  \rightarrow \sF\otimes \Sigma_{\bar i} \to 0.$$
 We get
$$H^0(\sF\otimes \Sigma_{\bar i})\to  H^1(\sF\otimes \Sigma_{i}^\vee)\to H^1(\sF\otimes \sO^{\oplus 2^{l}})$$
There is a graded subspace $V'$ of $H^0_*(\sF\otimes \Sigma_{\bar i} )$ which is mapped isomorphically to
$V \subset H^1_*(\F\otimes \Sigma_{i}^\vee)$.
This induces a map $\alpha: V'\otimes_k \Sigma_{\bar i}^\vee\to \sF.$\\
 Thus we can construct the following commuting diagram
$$
\begin{array}{cccccccc}

 0\to & \sF\otimes \Sigma_{ i}^\vee &\xrightarrow{ 1\otimes v_i^\vee} &\sF\otimes \sO^{\oplus 2^{l}}
 &\xrightarrow{1\otimes u_i^\vee} &\sF\otimes \Sigma_{\bar i}\to 0  \\
    &\uparrow{\alpha\otimes 1}   &            &  \uparrow{\alpha\otimes 1} &
    & \uparrow{\alpha\otimes 1} &\\
 0\to&(V'\otimes_k \Sigma_{\bar i}^\vee)\otimes \Sigma_{i}^\vee &\xrightarrow{ 1\otimes
 v_i^\vee}
 &(V'\otimes_k \Sigma_{\bar i}^\vee) \otimes \sO^{\oplus 2^{l}}
&\xrightarrow{ 1\otimes u_i^\vee} &(V'\otimes_k \Sigma_{\bar i}^\vee)\otimes \Sigma_{\bar i}\to 0
\end{array}
$$

Then $H^1_*(\alpha \otimes 1): H^1_*((V'\otimes_k\Sigma_{\bar i}^\vee) \otimes \Sigma_{ i} ^\vee) \to
H^1_*(\sF \otimes \Sigma_{ i} ^\vee)$ gives $V' \cong V$.

\end{proof}

\begin{corollary} Let  $\sF$ be a vector bundle on $\sQ_n$. Then any graded vector subspace $V$ of $\Sigma_i$-socle elements in
$H^1_*(\sF \otimes \Sigma_{ i}^\vee)_{soc}$ is an $A$-submodule of $H^1_*(\sF \otimes \Sigma_{ i}^\vee)_{soc}$.
\end{corollary}
\begin{proof} In proof above, $H^1_*(\alpha \otimes 1_{\Sigma_{i} ^\vee})$ is an $A$-module homomorphism, and by Lemma
\ref{ot}, the $A$-module $H^1_*((V'\otimes_k\Sigma_{\bar i}^\vee) \otimes \Sigma_{i} ^\vee)$ has the trivial $A$-module structure
where multiplication by graded elements in $A$ of positive degree is zero .
\end{proof}

For any vector bundle $\sE$ on $\sQ_n$, we will define invariants as follows:

\begin{definition}{\bf{(Horrocks Invariants of $\sE$)}} Let $\sE$ be a vector bundle on $\sQ_n$. It has a minimal associated Horrocks
datum  $(\sF_{min}, \beta)$. Let $V_i =
\ker H^1(\beta \otimes Id_{\Sigma_{i}^\vee}): H_*^1(\sF_{min}\otimes \Sigma_{i}^\vee) \to H^1_*(\sE\otimes \Sigma_{i}^\vee)$. Then
$V_i$ is a graded subspace of $H^1_*(\sF_{min} \otimes \Sigma_{ i}^\vee)_{soc}$.
The collection $(\sF_{min}, V_i)$ will be called Horrocks invariants for $\sE$. (As usual, when $n$ is even, this means
$(\sF_{min}, V_1,V_2)$ and when $n$ is odd, it means $(\sF_{min}, V_1)$.)
\end{definition}

\begin{remark}\end{remark}
\begin{enumerate}
\item $\sE$ is ACM if and only if $\sF_{min}$ is the zero bundle.  $V_i = 0$ as well.
\item In general, $V_i=0, \forall i$ if and only if $\sE$ is a direct sum of an Horrocks data bundle and an ACM bundle.
\item If $\sB$ is an ACM bundle, then $\sE$ and $\sE\oplus \sB$ will have the same Horrocks invariants.
\item If $(\sF_{min}, \beta, V_i)$ is a collection of Horrocks invariants for $\sE$ and $\phi$ is an automorphism of $\sF_{min}$, then
$\phi$ can be used to change $\beta: \sF_{min} \to \sE$ and hence also $V_i$ to get a new collection of Horrocks invariants for
$\sE$.
\item The definition could have used an arbitrary Horrocks data bundle $\sF$ for $\sE$ instead of the minimal one $\sF_{min}$
since $H^1_*(\Sigma_{i}^\vee)=0$ and hence the description of $V_i$ would not change.
\end{enumerate}

A stronger existence theorem for quadrics can now be stated than was proved in Theorem \ref{existence}. Below we here have a
statement that deals with arbitrary subspaces of socle elements.

\begin{theorem}{\bf(Existence)} Let $\sF_{min}$ be a minimal Horrocks data bundle on $\sQ_n$ and  let $V_i$ be a
graded vector subspace of $H^1_*(\sF_{min} \otimes \Sigma_{ i}^\vee)_{soc}$. Then there exists a vector bundle $\sE$ with the
Horrocks invariants $(\sF_{min},V_1,V_2)$ (when $n$ is even) and invariants $(\sF_{min},V_1)$ (when $n$ is odd).
\end{theorem}
\begin{proof} We follow the approach in Theorem \ref{existence}. For notational convenience, assume $n$ is even, so $i=1,2$.
Let $\sB = (V_1\otimes_k \Sigma_1) \oplus (V_2 \otimes_k \Sigma_2)$. As in the earlier proof, we obtain a short exact sequence
(defining $\sE$):
$$ 0 \to \sF_{min} \xrightarrow{\beta} \sE \xrightarrow{\rho} (V_1\otimes_k \Sigma_1) \oplus (V_2 \otimes_k \Sigma_2) \to 0,$$
where $(\sF_{min},\beta)$ is a Horrocks datum for the bundle $\sE$ so obtained. Our goal is now to show that the image of
$H^0_*(\sB \otimes \Sigma_i^\vee) \to H^1_*(\sF_{min}\otimes \Sigma_i^\vee)$ is $V_i$, whereas in the earlier proof, we showed
that it contained $V_i$.
Let $\Sigma_j(a)$ be any summand in $\sB$, and pick a non-zero section $s \in H^0(\Sigma_j(a)\otimes \Sigma_i^\vee(b)$, or a map
$\Sigma_i(-b) \xrightarrow{s} \Sigma_j(a)$. Then $a+b\geq 0$. $s\in H^0(\sB \otimes \Sigma_i^\vee(b))$ maps to zero in
$H^1_*(\sF_{min}\otimes \Sigma_i^\vee)$ iff the pullback of the short exact sequence by the map $s:\Sigma_i(-b) \xrightarrow{}
\sB$ is a split sequence.  If $a+b>0$, by Lemma \ref{ot}, the map $\Sigma_i(-b) \xrightarrow{s} \Sigma_j(a)$ factors through
$\sO^{\oplus 2^{l}}(a)$. The pullback of the short exact sequence by the map $\sO^{\oplus 2^{l}}(a) \to
\Sigma_j(a) \subseteq \sB$ splits since the extension is defined by socle elements. Hence so does the pullback by the map
$\Sigma_i(-b) \to \Sigma_j(a)\subseteq \sB$.

It follows that
the only non-zero contribution from this summand $\Sigma_j(a)$ to the image of $H^0(\sB \otimes \Sigma_i^\vee(b))$ occurs when $a+b=0$. If
$i \neq j$, $\Hom(\Sigma_i,\Sigma_j) =0$ and so no section $s$ can be found. If $i=j$, $\End(\Sigma_i)=k$ and it follows
that the image of $s$ lies in $V_i$. Thus the image of $H^0_*(\sB \otimes \Sigma_i^\vee)$ is exactly $V_i$.
\end{proof}

As pointed out after Theorem \ref{existence}, if $\sF_{min}$ has a $\Psi$-sequence $0 \to \sF_{min} \to \sP^0 \to \sG_{min} \to 0$, then
the $\sE$ constructed in the above theorem has $\gamma$-sequence given as
$$ 0 \to \sE \to \oplus_i (V_i \otimes_k \Sigma_i) \oplus \sP^0 \to \sG_{min} \to 0.$$
It is also easy to see that since $\sF_{min}$ has no summands of type $\Sigma_i$, neither does $\sE$. Conversely, suppose
$\sE$ is a vector bundle on $\sQ_n$ with  Horrocks invariants $(\sF_{min},V_i)$ and with no summands of type $\Sigma_i$. It
will follow from the next theorems that $\sE$ has a $\gamma$-sequence with $\sA_{\sE} = \oplus_i (V_i \otimes_k \Sigma_i)
\oplus \sP'$, where $\sP'$ is free.

The following two uniqueness results follow easily from the general theorems of Section 1.

\begin{theorem}\label{uniqueness}{\bf(Uniqueness)}\
Given $\sE,\sE'$ two bundles on $\Q_n$ without ACM summands, with Horrocks invariants $(\sF_{min}, V_i)$, $(\sF'_{min},
V_i')$.
 Suppose $\exists \phi: \sF_{min} \cong \sF'_{min}$, such that the induced isomorphisms
 $H^1_*(\sF_{min} \otimes \Sigma_{i}^\vee)\cong H^1_*(\sF'_{min} \otimes \Sigma_{i}^\vee)$ carry $V_i$ to $V_i'$ for each $i$.
 Then $\sE$ and $\sE'$ are isomorphic.
\end{theorem}

\begin{proof}
We may assume that $\sE$ and $\sE'$ have the same minimal Horrocks data bundle $\sF_{min}$.
If $\sF_{min}$ is zero,  $\sE, \sE'$ are ACM and the theorem does not apply. So we will assume that $\sF_{min}$ is a
non-free minimal
Horrocks data bundle. If $V_i$ are $0$ for $i=1,2$, then $\sE$ is stably equivalent to $\sF_{min}$, and being without
ACM summands,
it must be isomorphic to $\sF_{min}$. Since $V_i'$ will also be zero, the same is true for $\sE'$ and we conclude that $\sE \cong
\sE'$. So assume $V_i$ is non-zero for some $i$. If there is an automorphism $\phi$ of $\sF_{min}$ which carries $V_i$ to $V'_i$, in
the diagram of proposition \ref{diag} for $\sE'$, we may replace $\beta': \sF_{min} \to \sE'$ by $ \beta'\circ\phi^{-1}$
\textit{etc.} and
assume that $\beta$ and $\beta'$ give the same kernel $V_i$ in $H^1_*(F_{min}\otimes \Sigma_{ i}^\vee)$.

We can now apply Theorem \ref{iso-thm} to conclude the result.
\end{proof}

\begin{theorem}\label{isomorphism} Let $\sE, \sE'$ be vector bundles on $\sQ_n$ with no ACM summands. Suppose
$\sigma:\sE \to \sE'$ is a homomorphism such that $\sigma$ induces $H^j_*(\sE)\cong H^k_*(\sE')$ for $1\leq j \leq n-1$ and
also isomorphisms $H^1_*(\sE\otimes \Sigma_i^\vee)\cong H^1_*(\sE'\otimes \Sigma_i^\vee)$ for all $i$.
Then $\sigma$ is an isomorphism.
\end{theorem}
\begin{proof} This is just Theorem \ref{iso-thm2} with the additional condition that $\sE$ has no ACM summands.
\end{proof}

\section{The Veronese Surface}

The Veronese surface ${\sV}\subset\mathbb P^5$ is an arithmetically Cohen-Macaulay embedding which is not arithmetically
Gorenstein. The study of vector bundles on $\sV$ is trivial if we view $\sV$ as $\mathbb P^2$. Below we discuss how the
techniques of section one apply to the embedded variety $\sV$. With its polarization from the embedding, $\sV$ has two
irreducible, non-free ACM bundles (up to twists).
Hence, as in the case of quadric hypersurfaces of even dimension, we can define Horrrocks invariants $(\sF_{min}, V, W)$
for any vector bundle $\sE$ on $\sV$. But unlike in the case of the quadric, where $V,W$ were independent of each other,
here there is a dependency between them.

In the following discussion, we will write $\sO_{\sV}(1)$ for
$\sO_{\mathbb P^5}(1)|_{\sV}$ and $\sO_{\sV}(n)$ for $\sO_{\sV}(1)^{\otimes n}$. We will write $\sL$ for $\sO_{\mathbb P^2}(1)$
and $\sU$ for
$\Omega^1_{\sV}\otimes \sL$. Then the only irreducible ACM bundles on ${\sV}$ (with respect to the polarization
$\sO_{\sV}(1)$)
are $\sO_{\sV}(n)$, $\sL(n)$ and $\sU(n)$. In the diagram of a bundle $\sE$ on ${\sV}$ in Theorem \ref{diag},  the terms
$\A_{\sE}$ and $\sK_{\sE}$ are
built out of these three types of irreducible ACM bundles. The vector bundle $\sG$ is a free bundle and the $\Psi$-sequence
is the sheafification of a free presentation of the $A$-module $H^1_*(\sE)$.
The connection between $\A_{\sE}$ and $\sK_{\sE}$, given by the
$\Delta$-sequence in the diagram of $\sE$, is controlled by the following canonical sequences:

\begin{equation}\label{cv1}0\to \sU\xrightarrow{u} 3\sO_{\sV}\xrightarrow{v}\sL\to 0\end{equation}
and
\begin{equation}\label{cv2}0\to 3\sU(-1) \oplus \sO_{\sV}(-1) \xrightarrow{} 9\sO_{\sV}(-1) \xrightarrow{}\sU\to 0\end{equation}
where the second can be simplified non-canonically to
\begin{equation}\label{cv3}0\to 3\sU(-1)  \xrightarrow{u'} 8\sO_{\sV}(-1) \xrightarrow{v'}\sU\to 0.\end{equation}

In addition, there is the canonical sequence
\begin{equation}\label{cv4}0\to \sO_{\sV}(-1)\rightarrow 3\sL(-1)\rightarrow\sU\to 0\end{equation}

The two uniqueness theorems of Section 1 apply in this setting, where given a bundle $\sE$ on $\sV$, we can construct Horrocks
invariants for $\sE$ as $(\sF_{min},  V, W)$, where $(\sF_{min}, \beta)$ is an Horrocks datum for $\sE$, $V = \ker
[H^1_*(\sF_{min} \otimes \sL^\vee) \to H^1_*(\sE \otimes \sL^\vee)]$ and $W = \ker [H^1_*(\sF_{min} \otimes \sU^\vee) \to
H^1_*(\sE \otimes \sU^\vee)]$.
Thus to complete the classification of bundles on $\sV$ by this method, it remains to get a description of
any constraints on $V \subseteq H^1_*(\sF \otimes \sL^\vee)$ and $W \subseteq H^1_*(\sF_{min} \otimes \sU^\vee)$, and to finally
show that given $(\sF_{min}, V, W)$ with this constraints, there exists a bundle $\sE$ with those invariants.

By Remark \ref{socle-remark}, $V$ is an $A$-submodule of $\sL$-socle elements in $H^1_*(\sF_{min} \otimes \sL^\vee)_{soc}$
and $W$ is an $A$-submodule of $\sU$-socle elements in $H^1_*(\sF_{min} \otimes \sU^\vee)_{soc}$. By the following lemma, there
is no distinction between the concepts of graded $A$-submodules and graded vector subspaces of socle elements.

\begin{lemma} For any vector bundle $\sF$ on $\sV$, in the $A$-module structure of
of $H^1_*(\sF \otimes \sL^\vee)_{soc}$ as well as of $H^1_*(\sF \otimes \sU^\vee)_{soc}$, multiplication by graded elements
in $A$ of positive degree is zero.
\end{lemma}
\begin{proof} Let $\eta \in H^1(\sF(d) \otimes \sL^\vee)_{soc}$, giving a short exact sequence $0 \to \sF(d) \to \sA \to \sL
\to 0.$ Consider multiplication by $x\in A$ of degree one, $\sL(-1) \xrightarrow{\cdot x} \sL$. The pull back by this map of
the short exact sequence (\ref{cv1}) is split since $H^1(\sU\otimes \sL^\vee(1))=0$. So $\sL(-1) \xrightarrow{\cdot x} \sL$ factors
through $3\sO_{\sV}$. By the definition of $\sL$-socle element, the pull back of $\eta$ by $3\sO_{\sV} \to \sL$ splits, hence
also the pullback of $\eta$ by $\sL(-1) \xrightarrow{\cdot x} \sL$. Thus $x\cdot \eta = 0$.

A similar proof works for an element $\eta \in H^1(\sF(d) \otimes \sU^\vee)_{soc}$. One notices that the pull back by
$\sU(-1) \xrightarrow{\cdot x} \sU$ of the short exact sequence (\ref{cv3}) is split because $H^1_*(\sU\otimes \sU^\vee)=
3k$ supported in $H^1(\sU\otimes \sU^\vee(-1))$.
\end{proof}

In the definition of $\sU$-socle elements for $\sF$, the non-canonical inclusion $\sU^\vee \hookrightarrow 8\sO_{\sV})(1)$
can be replaced by a canonical composite inclusion $\sU^\vee \hookrightarrow 3\sL^\vee (1) \hookrightarrow 9\sO_{\sV})(1)$.
For any bundle $\sF$, this gives a canonical map
$$ \phi_{\sF}: H^1_*(\sF \otimes \sU^\vee)_{soc} \to 3H^1_*(\sF(1) \otimes \sL^\vee)_{soc}.$$

When $\sE$ is a vector bundle with Horrocks invariants $(\sF_{min}, V,W)$, it is immediate to see that
$V$ and $W$ are related by $\phi_{\sF_{min}}(W) \subseteq 3V(1)$.
This is a dependency between $V$ and $W$.  In fact, this is the only
requirement on the pair $(V,W)$ for proving an existence theorem on the Veronese surface:

\begin{theorem} \label{existence-V}Let $\sF_{min}$ be a minimal Horrocks data bundle on $\sV$, and let $V,W$ be
graded vector subspaces of
$H^1_*(\sF_{min} \otimes \sL^\vee)_{soc}, H^1_*(\sF_{min} \otimes \sU^\vee)_{soc}$ with the property that
$\phi_{\sF_{min}}(W) \subseteq 3V(1)$. Then there is a vector bundle $\sE$
on $\sV$ with Horrocks invariants $(\sF_{min},V,W)$.
\end{theorem}
\begin{proof} Construct $\sE$ as an extension of $\sF_{min}$ by $\sB = (V\otimes_k \sL) \oplus (W\otimes_k \sU)$:
$$ 0 \to \sF_{min} \xrightarrow{\beta} \sE \to \sB \to 0. \qquad (*)$$
Since $V$, $W$
are subspaces of socle elements, $\sE$ has $(\sF_{min},\beta)$ as its Horrocks datum. We wish to understand the images of
$H^0_*(\sB \otimes \sL^\vee) \to H^1_*(\sF_{min} \otimes \sL^\vee)$ and
$H^0_*(\sB \otimes \sU^\vee) \to H^1_*(\sF_{min} \otimes \sU^\vee)$.
$\End (\sL)= \End (\sU)=k$ and the image of $V\cdot I_{\sL} \subseteq H^0(V \otimes \sL \otimes \sL^\vee)$ and
$W \cdot I_{\sU} \subseteq H^0(W \otimes \sU \otimes \sU^\vee)$ give $V$ and $W$ in $H^1_*(\sF_{min} \otimes \sL^\vee)_{soc}$
and $H^1_*(\sF_{min} \otimes \sU^\vee)_{soc}$. It remains to analyze any other contributions to the two images inside
$H^1_*(\sF_{min} \otimes \sL^\vee)_{soc}$ and $H^1_*(\sF_{min} \otimes \sU^\vee)_{soc}$ and prove that the images are just
$V$ and $W$ respectively.

Let  $\sL(b),\sU(b)$ be any summands in  $(V\otimes_k \sL) \oplus (W\otimes_k \sU)$. Consider maps
$\sL(a) \xrightarrow{\sigma_1} \sL(b)$,  $\sL(a) \xrightarrow{\sigma_2} \sU(b)$,
$\sU(a) \xrightarrow{\sigma_3} \sU(b)$, $\sU(a) \xrightarrow{\sigma_4} \sL(b)$. For $\sigma_1$, assume $a<b$ since we wish
to omit endomorphisms of $\sL$. Likewise for $\sigma_3$. In the sequence (\ref{cv1}) tensored by $\sL^\vee(b-a)$,
$H^1(\sU\otimes \sL^\vee(b-a))=0$ and in the sequence (\ref{cv3}) tensored by $\sU^\vee(b-a)$,
$H^1(3\sU(-1)\otimes \sU^\vee(b-a))=0$. Hence $\sigma_1$ factors through $3\sO_{\sV}(b)$ and $\sigma_3$ factors through
$8\sO_{\sV}(b-1)$. By the socle nature of the extension (*), pullbacks of the (*) by $\sigma_1, \sigma_3$ split, hence the
element $\sigma_1 \in H^0(\sL(b)\otimes \sL^\vee(-a))$ maps to zero in $H^1_*(\sF_{min} \otimes \sL^\vee)$, and likewise
$\sigma_3$ maps to zero in $H^1_*(\sF_{min} \otimes \sU^\vee)$.

For $\sigma_4$ to be non-zero, we require that $a<b+1$. We know that $H^1(\sU \otimes \sU^\vee(b-a)) = 0$. Hence the same argument
applies to show that $\sigma_4$ factors through $3\sO_{\sV}(b)$ and we are done. The arguments for $\sigma_3, \sigma_4$
show that the image of
$H^0_*(\sB \otimes \sU^\vee) \to H^1_*(\sF_{min} \otimes \sU^\vee)$ equals $W$.

For $\sigma_2$ to be non-zero, we require that $a<b$ and we know that $H^1(3\sU(-1)\otimes \sL^\vee(b-a))=0$ except when
$b-a=1$. Hence the only situation of difficulty is when we have $\sigma_2: \sL(b-1) \to \sU(b)$. Suppose the pullback of
our short exact sequence (*) by $\sL(b-1) \xrightarrow{\sigma_2}\sU(b) \hookrightarrow \sB$ is non-split. The pullback
of (*) by $\sU(b) \hookrightarrow \sB$ gives a non-zero element $w$ of degree $-b$ in $W \subseteq
H^1_*(\sF_{min}\otimes \sU^\vee)_{soc}$.
The non-split pullback by $\sL(b-1) \to \sB$ gives a non-zero element $v$ in $H^1(\sF_{min}\otimes \sL^\vee(-b+1))_{soc}$
which is the image of $w$ under $\sigma_2^\vee$. Since $\sigma_2^\vee$ is one component in $\sU^\vee(-b)
\hookrightarrow 3\sL^\vee(-b+1)$,  the assumption that
$\phi_{\sF_{min}}(W) \subseteq 3V(1)$ tells us that $v \in V$. Thus, the image of $H^0_*(\sB \otimes \sL^\vee) \to
H^1_*(\sF_{min}\otimes \sL^\vee)$ equals $V$.

\end{proof}

We conclude with an example.
\begin{example} \end{example} The simplest non-ACM bundle on $\sV$ is $\sE = \Omega^1_{\sV} = \sU\otimes \sL^{\vee}$
with $H^1_*(\sE) =k$ and $\gamma$-sequence $0 \to \sE \to 3\sL^\vee \to \sO_{\sV} \to 0$, while its minimal Horrocks
data bundle is $\sF =\sF_{min}= \Omega^1_{\mathbb P^5}|_{\sV}$ with $\Psi$ sequence $0 \to \sF \to 6\sO_{\sV}(-1)
\to \sO_{\sV} \to 0$. The map $\beta: \sF \to \sE$ is the standard map $\Omega^1_{\mathbb P^5}|_{\sV} \to \Omega^1_{\sV}$
which is a surjective map of vector bundles but not surjective on the module of global sections. The Horrocks invariants
$(\sF,V,W)$ of $\sE$ are easy to work out and are described below.

$H^1_*(\sF \otimes \sL^\vee) = H^1(\sF(1)\otimes \sL^\vee) = 3k$, and $H^1_*(\sE \otimes \sL^\vee)=0$, hence $V = 3k
= H^1(\sF(1) \otimes \sL^\vee)$, where all elements in $H^1_*(\sF \otimes \sL^\vee)$ are $\sL$-socle.

There is a commutative diagram that shows the only non-zero parts of $H^1_*(\sF\otimes \sU^\vee)$ and
$H^1_*(\sE\otimes \sU^\vee)$:
 \[ \begin{CD}
 H^0(\sU^\vee) &\hookrightarrow &H^1(\sF\otimes \sU^\vee)               &\to &H^1(6 \sU^\vee(-1)) &\to 0 \\
 ||             &               & \downarrow{\beta\otimes I_{\sU^\vee}} &     &    \downarrow       &   \\
 H^0(\sU^\vee) & \cong          &H^1(\sE\otimes \sU^\vee)               & \to & 0                 &
 \end{CD} \]
Hence $H^1_*(\sF\otimes \sU^\vee)= H^1(\sF\otimes \sU^\vee)$ is nine-dimensional, and the the kernel $W$ of
$H^1_*(\beta\otimes I_{\sU^\vee})$ is a six-dimensional subspace (of $\sU$-socle elements)that maps isomorphically
to $H^1(6 \sU^\vee(-1))$.

When we apply the construction of the existence theorems \ref{existence}, \ref{existence-V} to the data $(\sF, V, W)$,
we obtain a vector bundle $\tilde\sE$ and a push-out diagram (refer to the discussion after Theorem \ref{existence}):
 \[ \begin{CD}
       &0                &      & 0                         &    &        &      \\
        &\downarrow      &      & \downarrow                &    &        &      \\
 0 \to &\sF              &\to   &6\sO_{\sV}(-1)             &\to &\sO_{\sV} &\to 0 \\
       &\downarrow{\tilde\beta}&      & \downarrow                &    & ||       &      \\
 0 \to &\tilde\sE        &\to   & \sA_{\tilde\sE}           &\to &\sO_{\sV} &\to 0 \\
        &\downarrow      &      & \downarrow                &    &         &      \\
        &\sB             &\cong & \sB                       &    &         &      \\
         &\downarrow      &      & \downarrow                &    &         &      \\
     &0                &      & 0                         &    &        &      \\
 \end{CD} \]
where $\sB = (V\otimes_k\sL) \oplus (W\otimes_k\sU)$.

According to the uniqueness theorems, $\sE$ is a rank two summand
of the rank 20 bundle $\tilde\sE$, with the remaining summand of $\tilde\sE$ consisting of ACM bundles. In this example, even
$\sA_{\tilde\sE}$ is not obvious because the middle short exact sequence is not split. Indeed, the middle sequence is the
push-out of the left sequence, hence it is split iff under $\sF \to 6\sO_{\sV}(-1)$, the image of the element
$\tau \in H^1(\sF\otimes \sB^\vee)$ is zero in $H^1(6\sO_{\sV}(-1)\otimes \sB^\vee)$. However, the components of
$\tau$ in each of the $\sU$-summands of $\sB$ generate the vector space $W \subset H^1(\sF\otimes \sU^\vee)$, and $W$ maps
isomorphically to $H^1(6 \sU^\vee(-1))$. Hence the image of $\tau$ is non-zero.

To understand $\tilde\sE$ and $\sA_{\tilde\sE}$, a little more work is needed.  The fact that $W$ maps
isomorphically to $H^1(6 \sU^\vee(-1))$ tells us that the middle short exact sequence contains 6 copies of the
canonical sequence (\ref{cv4}). Hence $\sA_{\tilde\sE} = 21 \sL^\vee$. The map $\sA_{\tilde\sE} \to \sO_{\sV}$ is now
easy to understand and shows that $\tilde\sE = \sE \oplus 18\sL^\vee$.

\end{document}